\documentclass[12pt]{amsart}
\usepackage{amsmath, amssymb, verbatim}
\usepackage{amsrefs}
\usepackage{pinlabel}
\usepackage{graphicx}
\usepackage[margin=2cm]{geometry}                

\newtheorem{thm}{Theorem}[section]

\newtheorem{lem}[thm]{Lemma}

\newtheorem{prop}[thm]{Proposition}

\theoremstyle{definition}

\theoremstyle{remark}
\newtheorem{rem}[thm]{Remark}

\numberwithin{equation}{section}

\newcommand{\al}{\alpha}

\newcommand{\de}{\delta}

\newcommand{\si}{\sigma}
\newcommand{\ga}{\gamma}

\newcommand{\La}{\Lambda}
\newcommand{\la}{\lambda}
\newcommand{\Si}{\Sigma}
\newcommand{\Om}{\Omega}
\newcommand{\om}{\omega}
\renewcommand{\th}{\theta}

\newcommand{\s}{\mathbf s}

\newcommand{\C}{\mathbb C}
\newcommand{\D}{\mathbb D}
\newcommand{\N}{\mathbb N}
\newcommand{\Z}{\mathbb Z}
\newcommand{\Q}{\mathbb Q}

\newcommand{\del}{\partial}
\newcommand{\Diff}{Diff}
\newcommand{\Dehn}{Dehn}

\newcommand{\co}{\thinspace\colon}
\newcommand{\x}{\times}

\newcommand{\hra}{\hookrightarrow}

\DeclareMathOperator{\lk}{lk}

\DeclareMathOperator{\id}{id}

\title[Stein fillable contact 3--manifolds and positive open books]
{Stein fillable contact 3--manifolds and\\positive open books of genus one}
\author{Paolo Lisca}

\begin{document}

\begin{abstract} 
A 2--dimensional open book $(S,h)$ determines a closed, oriented 3--manifold $Y_{(S,h)}$
and a contact structure $\xi_{(S,h)}$ on $Y_{(S,h)}$. The contact structure $\xi_{(S,h)}$ is Stein fillable if  
$h$ is {\em positive}, i.e.~$h$ can be written as a product of right--handed Dehn twists. Work of 
Wendl implies that when $S$ has genus zero the converse holds, that is 
\begin{equation}\label{e:1}\tag{$*$}
\xi_{(S,h)}\ \text{Stein fillable}\quad \Longrightarrow \quad h\ \text{positive}. 
\end{equation}
On the other hand, results by Wand~\cite{Wa09} and by Baker, Etnyre and Van Horn--Morris~\cite{BEV10} 
imply the existence of counterexamples to~\eqref{e:1} with $S$ of arbitrary genus strictly greater than one. 
The main purpose of this paper is to prove~\eqref{e:1} under the assumption that $S$ is a one--holed torus 
and $Y_{(S,h)}$ is a Heegaard Floer $L$--space. 
\end{abstract}

\maketitle

\section{Introduction}\label{s:intro} 

A {\em Stein surface} can be defined as a triple $(W,J,\varphi)$, where $W$ is a smooth, 
non--compact 4--manifold, $J$ is an integrable complex structure on $W$ 
(viewed as a bundle automorphism $J\co TW\to TW$) and $\varphi\co W\to [0,+\infty)$ is a smooth, proper function 
such that, setting $\la:=-J^* d\varphi\in\Om^1(W)$, the exact 2--form $\om_\varphi := d\la\in\Om^2(W)$ 
is everywhere non--degenerate, hence an exact symplectic form on $W$. A basic example is the 
triple $(\C^2, J_0, \sum_{i=1}^2 |z_i|^2)$, where $J_0$ is the standard complex structure on $\C^2$. 
If $c\in (0,+\infty)$ is a regular value of $\varphi$, the sublevel set $W_c:=\varphi^{-1}([0,c])$ is usually called a  
{\em Stein 4--manifold with boundary}. The restriction $\la_c:=\la|_{TW_c}\in\Om^1(\del W_c)$ satisfies 
$\la_c\wedge d\la_c >0$; in other words, the 2--plane distribution 
$\xi_{\del W_c}:=\ker(\la_c)\subset T\del W_c$ consisting of the complex lines tangent to $\del W_c$ 
is a positive contact structure on the oriented 3--manifold $\del W_c$. For more details on the basic notions 
in symplectic and contact topology recalled in this introduction we refer the 
reader to the book~\cite{Ge08} and the references therein. 

A contact 3--manifold $(Y,\xi)$ is called {\em Stein fillable} if it is orientation--preserving diffeomorphic 
to a pair $(\del W_c,\xi_{\del W_c})$ as above. In this situation we might simply say that {\em $\xi$ is a Stein 
fillable contact structure}. A typical source of Stein fillable contact structures is given by positive open books, 
defined below. 

An {\em abstract open book} is a pair $(\Si, h)$, where $\Si$ is an oriented surface with $\del\Si\neq\emptyset$ 
and $h$ is an element of the group $\Diff^+(\Si,\del\Si)$ of orientation--preserving diffeomorphisms of 
$\Si$ which restrict to the identity on the boundary. We will abusively confuse a differomorphism such 
as $h$ with its isotopy class modulo isotopies which fix $\del\Si$ pointwise. To the open book $(\Si,h)$ 
one can associate a closed, oriented 3--manifold $Y_{(\Si,h)}$ by taking the natural filling of the 
mapping cylinder of $h$: 
\[
Y_{(\Si,h)} := \Si\x [0,1]/(p,1) \sim (h(p),0) \cup_{\del} \del\Si\x D^2
\] 
The link $L:=\del \Si\x \{0\}\subset Y_{(\Si,h)}$ is fibered, with fibration $\pi\co Y_{(\Si,h)}\setminus L\to S^1$ 
given by the obvious extension of the natural projection 
\[
\Si\x [0,1]/(p,1) \sim (h(p),0)\to S^1=[0,1]/1\sim 0. 
\]
The pair $(L,\pi)$ is an {\em open book decomposition} of $Y_{(\Si,h)}$ with {\em binding} $L$ and 
pages $\Si_\th:= \overline{\pi^{-1}(\th)}$, $\th\in S^1$. 
The 3--manifold $Y_{(\Si,h)}$ carries a contact form $\la$ such that $\la|_L>0$ and $d\la|_{\Si_\th}>0$ for 
each $\th\in S^1$, with the contact structure $\xi_{(\Si,h)}=\ker\la\subset TY_{(\Si,h)}$ uniquely determined 
up to diffeomorphisms by the conjugacy class of $h$ in $\Diff^+(\Si,\del\Si)$.  Moreover, the 
map $(\Si, h)\mapsto (Y_{(\Si,h)}, \xi_{(\Si,h)})$ is surjective but not injective~\cite{Gi02}. 

We say that $h$ is {\em positive} if either $h=\id_\Si$ or $h=\de_{\ga_1}\cdots\de_{\ga_k}$, where 
$\ga_i\subset\Si$, $i=1,\ldots, k$, is a simple closed curve and $\de_{\ga_i}\in Diff^+(\Si,\del\Si)$ 
is a right--handed Dehn twist along $\ga_i$. 
We denote by $\Dehn^+(\Si,\del\Si)\subseteq\Diff^+(\Si,\del\Si)$ the monoid of positive, orientation--preserving 
diffeomorphisms of the pair $(\Si,\del\Si)$. When $h\in\Dehn^+(\Si,\del\Si)$ we say that the open 
book $(\Si,h)$ is {\em positive}. By~\cites{LP01, Gi02} (see also~\cites{AO01, AO01err} and~\cite{Pl04}*{Appendix~A}) 
we have the well--known fact that if $h\in\Dehn^+(\Si,\del\Si)$ then $\xi_{(\Si,h)}$ is Stein fillable, which leads 
naturally to the following Basic Question: 
\begin{equation}\label{e:bq}
 \xi_{(\Si,h)}\ \text{Stein fillable} \ \Longrightarrow\ h\in\Dehn^+(\Si,\del\Si)\ ?
\end{equation}
By~\cite{We10}, it is known that the answer to~\eqref{e:bq} is `yes' when $\Si$ is a planar 
surface, while in~\cite{Wa09} and~\cite{BEV10} are constructed examples with $g(\Si)=2$ 
for which the answer to~\eqref{e:bq} is `no'. Moreover, John Etnyre has observed~\cite{Epc13} that the examples 
of~\cites{Wa09, BEV10} can be used to easily construct similar examples for any genus 
$g(\Si)\geq 3$. We included a short sketch of his argument in Remark~\ref{r:highergenus}. 

The purpose of this paper is to prove Theorem~\ref{t:main} below, which shows that 
the answer to~\eqref{e:bq} is positive when $g(\Si)=1$, $\Si$ has connected boundary and $Y_{(\Si,h)}$ is a 
Heegaard Floer $L$--space. Recall that a closed, oriented 3--manifold $Y$ is a Heegaard Floer $L$--space, 
or simply an {\em $L$--space}, if $Y$ is a rational homology 3--sphere such that the rank of the Heegaard Floer group 
$\widehat{HF}(Y;\Z)$ (defined in~\cite{OS04}) equals the order of the finite group $H_1(Y;\Z)$. It is a well--known 
fact that the simplicity of the Heegaard Floer groups of an $L$--space $Y$ makes it possible, in certain situations, 
to gather useful information about the Stein fillings of $Y$ (cf.~\cites{OSgb04, Ba08, GLL11}). We will exploit 
this fact to prove the following. 
\begin{thm}\label{t:main} 
Let $T$ be an oriented, one--holed torus, $h\in\Diff^+(T,\del T)$, and suppose that 
$Y_{(T,h)}$ is a Heegaard Floer $L$--space. Then, 
\[
\xi_{(T,h)}\ \text{Stein fillable} \ \Longrightarrow\ h\in\Dehn^+(T,\del T).
\] 
\end{thm} 

The sub--monoid $\Dehn^+(T,\del T)\subset\Diff^+(T,\del T)$ and the Basic Question~\ref{e:bq}
were also considered in~\cite{HKM08}. In~\cite{HKM09} the authors gave a characterization 
of the elements $h\in\Diff^+(T,\del T)$ such that $\xi_{(T,h)}$ is a tight contact structure. 
The proof of Theorem~\ref{t:main} provides an explicit characterization of the elements 
$h\in\Dehn^+(T,\del T)$ such that $Y_{(T,h)}$ is a Heegaard Floer $L$--space (see 
the statements of Proposition~\ref{p:steinfill} and~\ref{t:newmain}). This should be compared with 
the known algorithm to establish the quasi--positivity of a closed 3--braid~\cite{Or04} (as explained 
in Section~\ref{s:background}, $\Diff^+(T,\del T)$ is isomorphic to the group of closed 3--braids). 

The paper is organized as follows. In Section~\ref{s:background} we recall some previously known 
results and we use them to show that Theorem~\ref{t:main} is implied by Theorem~\ref{t:newmain}. 
In Section~\ref{s:positivity} we prove the first half of Theorem~\ref{t:newmain}, and in Sections~\ref{s:construction} 
and~\ref{s:final} we prove the second half. 

{\bf Acknowledgements:} the author wishes to thank John Etnyre for pointing 
out Remark~\ref{r:highergenus}. The present work is part of the author's activities within CAST, a Research Network
Program of the European Science Foundation, and the PRIN--MIUR research project 2010--2011 ``Variet\`a 
reali e complesse: geometria, topologia e analisi armonica''. 

\section{Recollection of previous results and a refinement of Theorem~\ref{t:main}}\label{s:background}

Let $x, y \in\Diff^+(T,\del T)$ be right--handed Dehn twists along two simple closed curves in $T$ 
intersecting transversely once. Then, $\Diff^+(T,\del T)$ is generated by $x$ and $y$ subject to the 
relation $xyx=yxy$. We shall denote by $\exp\co\Diff^+(T,\del T)\to\Z$ the ``exponent--sum'' homomorphism 
defined on an element $h$ by first writing $h$ as a product of powers of $x, y$ and then taking $\exp(h)$ 
to be the sum of the exponents of $x$ and $y$ appearing in the product. This is a good definition because two 
such factorizations of $h$ are obtained from each other via finitely many applications of the homogeneous 
relation $xyx=yxy$. It is possible to check that there is an isomorphism from the 3--strand braid group $B_3$
onto $\Diff^+(T,\del T)$ sending $\si_1$ to  $x$ and $\si_2$ to $y$, where $\si_i\in B_3$, $i=1,2$, 
are the standard generators. Such isomorphism can be realized geometrically by viewing $T$ as a 
two--fold branched cover over the $2$--disk with three branch points: elements of $B_3$, viewed 
as automorphisms of the triply--pointed disk lift uniquely to elements of $\Diff^+(T,\del T)$. In our 
notation a product $\si_1\si_2\in B_3$, when viewed as composition of automorphisms, should be 
interpreted as {\em first} applying $\si_1$ and {\em then} $\si_2$. For this reason, throughout the paper, 
when we write the composition of two elements $\phi,\psi\in\Diff^+(T,\del T)$ as $\phi\psi$  we shall mean 
$\phi$ {\em followed by} $\psi$. The classification of $3$--braids due to Murasugi~\cite{Mu74} implies 
that each element of $\Diff^+(T,\del T)$ is conjugate to one of the following: 
\begin{itemize} 
\item 
$(xy)^{3d} x^{-m} y^{-1}$, $d\in\Z$, $m\in\{1,2,3\}$; 
\item 
$(xy)^{3d} y^m$, $d\in \Z$, $m\in\Z$; 
\item
$(xy)^{3d} x^{a_1} y^{-b_1}\cdots x^{a_n} y^{-b_n}$, $a_i, b_i, d\in\Z$, $a_i, b_i, n\geq 1$.  
\end{itemize}

The following statement is proved by combining results from~\cites{OSgb04, OS05, Pl04, Ba08}.
\begin{prop}\label{p:steinfill}
Let $h\in\Diff^+(T,\del T)$, suppose that $Y_{(T,h)}$ is a Heegaard Floer $L$--space and 
that $(W,J)$ is a Stein filling of $\xi_{(T,h)}$. Then, $c_1(W,J)=0$, $b_2^+(W)=0$, 
$b_2^-(W) = \exp(h) - 2$ and $h$ is conjugate to one of the following: 
\begin{enumerate} 
\item 
$(xy)^{3d} x^{-m} y^{-1}$, $d\in\{1,2\}$, $m\in\{1,2,3\}$; 
\item 
$(xy)^3 y^m$, $m\geq -4$; 
\item
$(xy)^3 x^{a_1} y^{-b_1}\cdots x^{a_n} y^{-b_n}$, $a_i,b_i\in\N$, $n\geq 1$, 
$\sum_{i=1}^n a_i + 4 \geq \sum_{i=1}^n b_i$.  
\end{enumerate}
Moreover, in the first two cases $h\in\Dehn^+(T,\del T)$. 
\end{prop} 

\begin{proof} 
By~\cite{OSgb04}*{Theorem~1.4} any symplectic filling $W$ of an $L$--space satisfies $b_2^+(W)=0$. 
The fact that $c_1(W,J)=0$ follows from the results of~\cite{Pl04}, as shown in the proof 
of~\cite{Ba08}*{Theorem~7.1}. It is a well--known fact that Stein 4--manifolds admit handle decompositions with only 
$0$--, $1$-- and $2$--handles. Since the assumption that $Y_{(T,h)}$ is an $L$--space implies $b_1(Y_{(T,h)})=0$ 
and a handle decomposition of $W$ can be viewed dually as obtained from $Y_{(T,h)}$ by attaching handles of 
index at least $2$, it follows that $b_1(W)=0$. Therefore, the Euler characteristic of $W$ satisfies $\chi(W) = 1 + b_2^-(W)$.  
Finally, combining~\cite{Ba08}*{Proposition~5.1} and~\cite{Ba08}*{Theorem~7.1}
we get $\exp(h) - 2 = \chi(W) - 1$, obtaining the first part of the statement.
In~\cite{Ba08} Baldwin determined the elements $h\in\Diff^+(T,\del T)$ such that the 3--manifold $Y_{(T,h)}$ 
is an $L$--space, as well as those such that the contact structure $\xi_{(T,h)}$ has non--vanishing 
contact invariant, a property which is always satisfied by Stein fillable contact structures~\cite{OS05}. 
The combination of Theorems~4.1 and~4.2 from~\cite{Ba08} together with the fact that $0\leq b_2^-(W)=\exp(h)-2$ 
immediately yields the fact that $h$ must be conjugate to one of the elements in $(1)$, $(2)$ or $(3)$ of the statement. 

To verify the last part of the statement, in Case (1) it clearly suffices to check that $(xy)^3 x^{-3} y^{-1}$ is positive.
Since a conjugate of either $x$ or $y$ is a right--handed Dehn twist, it is enough to express this element 
as a product of conjugates of $x$ and $y$.  
Indeed, using the relation $xyx=yxy$ it is easy to verify that 
\[
(xy)^3 x^{-3} y^{-1} = yx^2 y x^{-1} y^{-1} = y ( x (xyx^{-1}))y^{-1} = 
yxy^{-1}\cdot (yx) y (yx)^{-1}.
\]
For Case (2), it suffices to check that $(xy)^3 y^{-4}$ is positive. 
As before, we just need the relation $xyx=yxy$:  
\[
xyxyxyy^{-4} = xyyxyyy^{-4} = x\cdot y^2 x y^{-2}. 
\]
\end{proof}

\begin{rem} 
Not all the elements of Case~$(3)$ in Proposition~\ref{p:steinfill} are in $\Dehn^+(T,\del T)$. 
For instance, the element $(xy)^3 x y^{-1} x y^{-5}$ satisfies the conditions of Case~$(3)$ but it is conjugate 
to the element considered in~\cite{HKM08}*{Subsection~2.5} and shown there to be non--positive. Many more 
such examples exist, as follows from Theorem~\ref{t:newmain} below. 
\end{rem} 

By Proposition~\ref{p:steinfill}, in order to establish Theorem~\ref{t:main} it suffices to prove its statement 
for the elements $h\in\Diff^+(T,\del T)$ conjugate to those of the form (3) in the proposition. In fact, 
we will prove a refinement of the statement of Theorem~\ref{t:main}, stated as Theorem~\ref{t:newmain} below, 
which gives a characterization of the positive elements. 

Now we need to introduce some notation in order to state Theorem~\ref{t:newmain}.
Let $\N$ be the set of (positive) natural numbers, and let $k\in\N$.
We say that $\hat z\in\N^{k+1}$ is a  {\em blow--up} of $z=(n_1,\ldots,n_k)\in\N^k$ 
if  
\[
\hat z = 
\begin{cases} 
(1,n_1+1,n_2,\ldots,n_{k-1},n_k+1),\ \text{or} \\
(n_1,\ldots,n_i+1,1,n_{i+1}+1,\ldots,n_k),\ \text{for some}\ 1\leq i < k,\ \text{or} \\
(n_1+1,n_2,\ldots,n_{k-1},n_k+1,1).  
\end{cases} 
\]
We will use the notation $\hat z\rightarrow z$ to denote the fact that $\hat z$ is a blowup of $z$, and 
the notation 
\[
(s_1,\ldots,s_N)\rightarrow\stackrel{\text{blowup}}{\cdots}\rightarrow (0,0).
\]
to indicate that the $N$--tuple $(s_1,\ldots,s_N)\in\N^N$ can be obtained from $(0,0)$ via a sequence of successive blowups.
For example, we have $(2,3,1,2,3,1)\rightarrow\stackrel{\text{blowup}}{\cdots}\rightarrow (0,0)$, because there is 
the sequence of blowups:  
\[
(2,3,1,2,3,1)\rightarrow (1,3,1,2,2)\rightarrow (2, 1, 2, 1)\rightarrow (1, 1, 1)\rightarrow (0,0).
\]
\begin{thm}\label{t:newmain}
Let $h=(xy)^3 x^{a_1} y^{-b_1}\cdots x^{a_n} y^{-b_n}\in \Diff^+(T,\del T)$,  $a_i,b_i,n\geq 1$, 
and\\ $N:=\sum_{i=1}^n b_i\geq 2$. If $N=1$ then $h\in\Dehn^+(T,\del T)$. 
If $N\geq 2$ then the following are equivalent: 
\begin{enumerate} 
\item
$h\in\Dehn^+(T,\del T)$;
\item
$(Y_{(T,h)}, \xi_{(T,h)})$ is Stein fillable;
\item 
There is a sequence of blowups $(s_1,\ldots,s_N)\rightarrow\stackrel{\text{blowup}}{\cdots}\rightarrow (0,0)$ such that, 
setting\\ $(c_1,\ldots, c_N):= (a_1+2,\overbrace{2,\ldots,2}^{b_1-1},a_2+2,\ldots,a_n+2,\overbrace{2,\ldots,2}^{b_n-1})$, 
we have
\[c_1\geq s_1,\ c_2\geq s_2,\ \ldots, c_N\geq s_N\] 
\end{enumerate} 
\end{thm} 

The proof of Theorem~\ref{t:newmain} will occupy the rest of the paper. More precisely, we already know that 
$(1)\Rightarrow (2)$. In Section~\ref{s:positivity} we show that $(3)\Rightarrow (1)$, 
and in the remaining sections we show that $(2)\Rightarrow (3)$. 

\section{Construction of positive diffeomorphisms}\label{s:positivity} 

Given any $N$--tuple $s = (s_1,\ldots, s_N)\in\N^N$, there is a unique way of writing $s$ as 
\[
s = (a_1+2,\overbrace{2,\ldots,2}^{b_1-1},\ldots, a_n+2, \overbrace{2,\ldots,2}^{b_n-1})
\]
for some integers $a_1,\ldots, a_n\geq -2$, $b_1,\ldots, b_n, n\geq 1$. We define 
\[
h(s) := (xy)^3 x^{a_1} y^{-b_1}\cdots x^{a_n} y^{-b_n} \in\text{Diff}^+(T, \del T).
\]
In this section we prove that $(3)\Rightarrow (1)$ in Theorem~\ref{t:newmain}. We start by proving this fact  
in the special case of $N$--tuples which are obtained from $(0,0)$ via a sequence of successive blowups.

\begin{lem}\label{l:blowups-id}
Suppose that $s\in\N^N$ is obtained from $(0,0)$ via a sequence of successive blowups 
in the sense of Section~\ref{s:background}. Then, $h(s) = \id_T \in\text{Diff}^+(T, \del T)$.
\end{lem} 

\begin{proof} 
Note that 
\[
(0,0) = (-2+2, \overbrace{2,\ldots,2}^{0=1-1}, -2+2, \overbrace{2,\ldots,2}^{0=1-1}), 
\]
hence 
\begin{multline*}
h((0,0)) = (xy)^3 x^{-2} y^{-1} x^{-2} y^{-1} = (xy)^3 x^{-1} (x^{-1} y^{-1} x^{-1}) x^{-1} y^{-1} 
\\
= (xy)^3 (x^{-1} y^{-1} x^{-1}) (y^{-1} x^{-1} y^{-1})  
= (xy)^3 y^{-1} x^{-1} y^{-1} x^{-1} y^{-1} x^{-1} = (xy)^3 (xy)^{-3} = \id_T
\end{multline*}
Since each element of $S$ by blowing--up $(0,0)$, to prove the lemma it suffices to check that, if 
$\hat s$ denotes a blow--up of $s\in S$, $h(s)$ and $h(\hat s)$ are 
conjugate in $\Diff^+(T,\del T)$ 
for every $s\in S$. We may write 
\[
s =(a_1+2,\overbrace{2,\ldots,2}^{b_1-1},\ldots, a_n+2, \overbrace{2,\ldots,2}^{b_n-1})
\]
for some  $a_1,\ldots, a_n\geq -2$, $b_1,\ldots, b_n\geq 1$ and $n\in\N$. Then, 
$h(s) = (xy)^3 x^{a_1} y^{-b_1}\cdots x^{a_n} y^{-b_n}$ 
and depending on how the blowup is performed, there are several possibilities for $\hat s$. 
These lead to the following possible cases for $h(\hat s)$: 
\[
h(\hat s) = 
\begin{cases} 
(xy)^3 x^{-1} y^{-1} x^{a_1+1} y^{-b_1}\cdots x^{a_n} y^{-b_n+1} xy^{-1} ,\\
(xy)^3 x^{a_1} y^{-b_1}\cdots x^{a_i+1} y^{-1} x^{-1} y^{-1} x y^{-b_i+1} x^{a_{i+1}} \cdots x^{a_n} y^{-b_n}, i\neq 1, n,\\
(xy)^3 x^{a_1+1} y^{-b_1}\cdots x^{a_n} y^{-b_n+1} x y^{-1} x^{-1} y^{-1}.
\end{cases} 
\]
It is straightfoward to check that in each case $h(\hat s)$ is conjugate to $h(s)$. In the first case, for instance, we have 
\begin{multline*}
h(\hat s) = \de x^{-1} y^{-1} x^{a_1+1} y^{-b_1}\cdots x^{a_n} y^{-b_n+1} xy^{-1} 
\sim \de yx(y^{-1}x^{-1} y^{-1}) x^{a_1+1} y^{-b_1}\cdots x^{a_n} y^{-b_n} 
\\ 
= \de y x x^{-1}y^{-1} x^{-1} x^{a_1+1} y^{-b_1}\cdots x^{a_n} y^{-b_n} = h(s)
\end{multline*}
We omit the easy verifications in the remaining cases. 
\end{proof} 
 
In order to establish the implication $(3)\Rightarrow (1)$ of Theorem~\ref{t:newmain} for general $N$--tuples, we first analyze 
what happens when a single entry of the $N$--tuple is increased by $1$. 

\begin{lem}\label{l:xinsert}
Let $s = (s_1,\ldots, s_N)\in\N^N$ and $s' = (s_1,\ldots, s_{i-1}, s_i+1,s_{i+1}, \ldots, s_N)$ for some $i\in\{1,\ldots, N\}$.   
Then, there are $\phi, \psi\in\Diff^+(T,\del T)$ such that $h(s) = \phi \psi$ and $h(s') = \phi x \psi$. 
\end{lem}

\begin{proof} 
Write $(s_1,\ldots, s_N) = (a_1+2,\overbrace{2,\ldots,2}^{b_1-1},\ldots, a_n+2, \overbrace{2,\ldots,2}^{b_n-1})$ 
for some integers $a_1,\ldots, a_n\geq -2$, $b_1,\ldots, b_n, n\geq 1$, so that  
$h(s) = (xy)^3 x^{a_1} y^{-b_1}\cdots x^{a_n} y^{-b_n}$. If $s_i=a_j+2$ for some $j$, then 
$h(s')$ is obtained from $h(s)$ by replacing $x^{a_j}$ with $x^{a_j+1}$, and the statement holds. 
If $s_i=2$, then it is easy to check that $h(s')$ is obtained from $h(s)$ by replacing 
$y^{-b_j}$, for some $j$, with $y^{-a} x y^{-b}$ where $a+b=b_j$. Again, the statement holds.
\end{proof} 

We are now ready to reach the goal of the section. 

\begin{prop}\label{p:3implies1}
Let $(c_1,\ldots, c_N)$ be an $N$--tuple of integers and suppose that there is a sequence of blowups 
$(s_1,\ldots,s_N)\rightarrow\stackrel{\text{blowup}}{\cdots}\rightarrow (0,0)$ such that 
$c_1\geq s_1,\ c_2\geq s_2,\ \ldots, c_N\geq s_N$.  Then, $h(c_1,\ldots, c_N)\in\Dehn^+(T,\del T)$. 
\end{prop}

\begin{proof} 
By Lemma~\ref{l:blowups-id} we have $h(s_1,\ldots,s_N) = \id_T$. In view of the inequalities 
$c_1\geq s_1,\ c_2\geq s_2,\ \ldots, c_N\geq s_N$, in order to prove the statement it clearly 
suffices to show that, if $s = (s_1,\ldots, s_N)\in\N^N$ and 
$s' = (s_1,\ldots, s_{i-1}, s_i+1,s_{i+1}, \ldots, s_N)$ for some $i\in\{1,\ldots, N\}$, 
\begin{equation}\label{e:+1pos}
h(s)\in\Dehn^+(T, \del T)\ \Longrightarrow\  h(s')\in\Dehn^+(T,\del T). 
\end{equation}
By Lemma~\ref{l:xinsert} there are $\phi, \psi\in\Diff^+(T,\del T)$ such that 
\[
h(s') = \phi x \psi = \phi \psi \psi^{-1} x \psi = h(s) (\psi^{-1} x \psi).
\]
By assumption $h(s)\in\Dehn^+(T, \del T)$, each conjugate of $x$ is in $\Dehn^+(T, \del T)$ and 
$\Dehn^+(T, \del T)$ is a monoid, so we conclude that~\eqref{e:+1pos} holds. 
\end{proof} 

\section{A topological construction}\label{s:construction}

The purpose of this section is to establish Proposition~\ref{p:isoemb}, which will be used in Section~\ref{s:final} 
to prove $(3)\Rightarrow (1)$ in Theorem~\ref{t:newmain}. We derive the proposition by applying Donaldson's 
celebrated theorem~\cite{Do87}*{Theorem~1} to certain suitably constructed smooth, closed 4--manifolds. 

Let $h$ be an element of $\Diff^+(T,\del T)$ factorized as in the statement of Theorem~\ref{t:newmain}:
\[
h=(xy)^3 x^{a_1} y^{-b_1}\cdots x^{a_n} y^{-b_n}\quad a_i, b_i, n\geq 1.
\] 
Define the string $(c_1,\ldots, c_N)$, where $N=\sum_{i=1}^n b_i$, by setting 
\begin{equation}\label{e:cstring}
(c_1,\ldots, c_N):=(a_1+2,\overbrace{2,\ldots,2}^{b_1-1},a_2+2,\ldots,a_n+2,\overbrace{2,\ldots,2}^{b_n-1})
\end{equation}
Note that if $N=\sum_{i=1}^n b_i = 1$ then $n= b_1=1$. In that case $h$ is clearly positive, 
therefore from now on we shall assume $N\geq 2$. 

Consider the 3--manifold $Y$ defined by performing integral Dehn surgeries on $S^3$ according to 
the framed link $L$ of Figure~\ref{f:1}. 
\begin{figure}[ht]
\labellist
\small\hair 2pt
\pinlabel $c_1$ at 150 92
\pinlabel $c_2$ at 335 92
\pinlabel $c_{N-2}$ at 900 92
\pinlabel $c_{N-1}$ at 1090 92
\pinlabel $c_N$ at 610 22
\endlabellist
\centering
\includegraphics[scale=0.3]{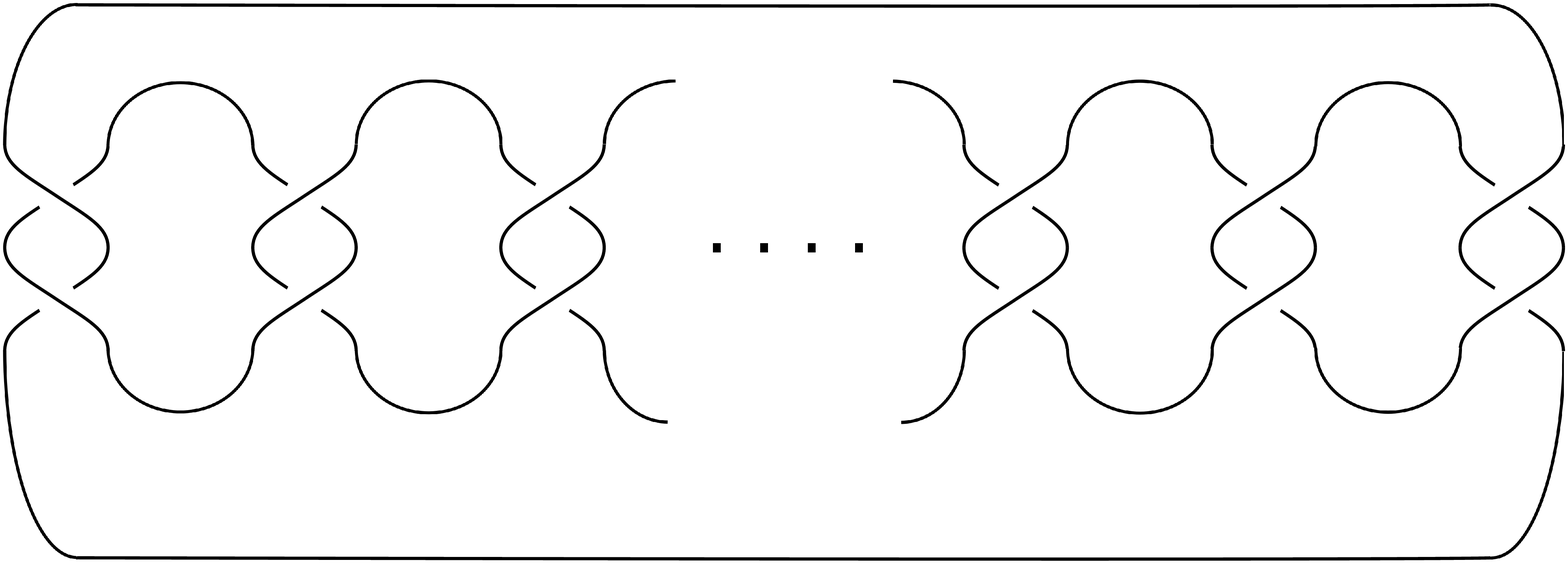}
\caption{Surgery presentation for $Y$ and handle decomposition of $X$.} 
\label{f:1}
\end{figure}
We are going to argue that $Y$ carries an open book decomposition with page 
a one--holed torus $S$ and monodromy $h$ when $S$ is suitably identified with $T$. 
In other words, $Y=Y_{(T,h)}$. Consider the picture on the 
left--hand side of Figure~\ref{f:2} for any $i\in\{2,\ldots,n\}$. 
The picture illustrates a one--holed torus $S$ embedded in the complement 
of the framed link $L$. 

\begin{prop}\label{p:obd}
$S$ is the page of an open book decomposition on $Y$ which, 
under a suitable identification of $S$ with $T$, has monodromy $h$. 
\end{prop}

\begin{proof}
The following proof is an adaptation to the present situation 
of the arguments given in~\cite{KM94}*{Appendix}.
The surface $S$ can be isotoped to the one--holed torus $S'$ illustrated in the picture on the right--hand side 
of Figure~\ref{f:2}. To see that, just think about the fact that the complement of the Hopf link in $S^3$ is a torus times 
an interval. Moreover, the isotopy takes the oriented curves $\mu,\la\subset S$ in the left--hand picture to 
$\la',-\mu'\subset S'$, respectively, illustrated in the right--hand picture. 
We could also enlarge slightly $S$ to $\tilde S$ and $S'$ to $\tilde S'$ so that $\del\tilde S=\del\tilde S'$.  We may identify 
$S$ and $S'$ with $T$ so that the isotopy from $S$ to $S'$ induces an orientation--preserving diffeomorphism 
$\phi\co T\to T$ prescribed, in terms of an oriented basis of $H_1(T;\Z)$, by the matrix 
$M=\left(\begin{smallmatrix}0 & -1\\1 & 0\end{smallmatrix}\right)\in SL(2;\Z)$. Since the generators 
$x,y\in\Diff^+(T,\del T)$ associated to the given oriented basis correspond, respectively, to 
$\left(\begin{smallmatrix}1 & 1\\0 & 1\end{smallmatrix}\right)$ and 
$\left(\begin{smallmatrix}1 & 0\\-1 & 1\end{smallmatrix}\right)$, 
one can easily check that $\phi=x^{-1}y^{-1}x^{-1}$. 
\begin{figure}[ht]
\labellist
\small\hair 2pt
\pinlabel $c_{i-1}$ at 25 22
\pinlabel $\la$ at 55 160 
\pinlabel $S$ at 65 100
\pinlabel $\mu$ at 225 160
\pinlabel $c_i$ at 220 22
\pinlabel $\mu'$ at 435 160
\pinlabel $c_{i-1}$ at 460 22
\pinlabel $\la'$ at 600 160
\pinlabel $S'$ at 595 100
\pinlabel $c_i$ at 650 22
\endlabellist
\centering
\includegraphics[scale=0.45]{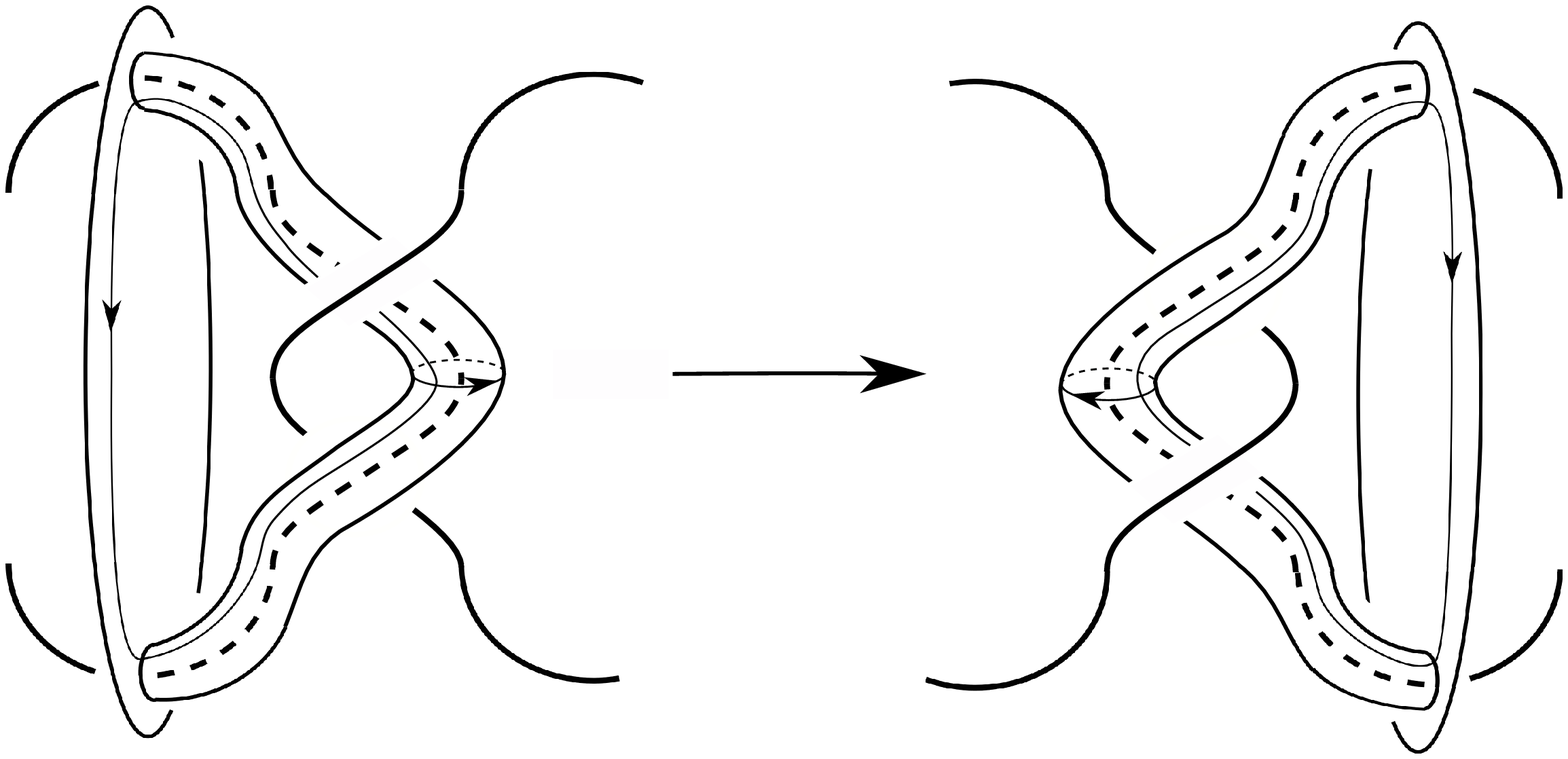}
\caption{The isotopy from $S$ to $S'$.} 
\label{f:2}
\end{figure}
The same analysis applies to every clasp of $L$ except the one between the $1$--st and the $N$--th components of $L$. 
In that case, an analysis as above shows that matrix associated to the last isotopy is $-M=M^{-1}$ 
instead of $M$, corrisponding to the diffeomorphism $\phi^{-1}=xyx$. 

Now we claim that, for each $i=1,\ldots, N$, there is another isotopy sending the surface $S' := D'\cup A'$, 
illustrated on the left--hand side of Figure~\ref{f:3}, to the surface $S=D\cup A$ 
illustrated on the right--hand side. This isotopy goes through the solid torus glued along a neighborhood 
of the $i$--th component of $L$. In fact, it fixes $D$ and sends $A'$ to $A$, sending the 
simple closed curve $\la'\subset S'$ to $\la''\subset S$, which twists $c_i$ times around $A$. 
To see this, recall that the presence of the framing coefficient ``$c_i$'' means that a neighborhood 
of the $i$--th component $L_i$ of $L$ with a meridian--longitude basis $m, \ell$ of its boundary is first 
cut out and then re--glued by sending $m$ to $c_i m+\ell$ and $\ell$ to $-m$. Thus, the simple closed curve 
$(\la'\cap A')\cup(\la''\cap A)$ bounds a meridional disk in the glued--up solid torus, while $A'$ and $A$ can be identified 
with neighborhoods of parallel longitudinal curves on its boundary. This means that $A'$ is isotopic to $A''$ via 
an isotopy which carries the annulus across the glued--up solid torus, sending $\la'\cap A'$ to $\la''\cap A$, 
and the claim is proved. As in the previous case of Figure~\ref{f:2}, we may identify $S'$ and $S$ with 
$T$ so that the isotopy induces an automorphism $T\to T$ represented by the matrix 
$N(c_i)=\left(\begin{smallmatrix} 1 & c_i\\0 & 1\end{smallmatrix}\right)\in SL(2;\Z)$, hence given by $x^{c_i}$. 
\begin{figure}[ht]
\labellist
\small\hair 2pt
\pinlabel $\mu'$ at 35 160
\pinlabel $c_{i-1}$ at 75 30
\pinlabel $A'$ at 138 280
\pinlabel $\la'$ at 203 220
\pinlabel $D$ at 195 100
\pinlabel $c_i$ at 255 30
\pinlabel $c_{i+1}$ at 342 30
\pinlabel $c_{i-1}$ at 605 30
\pinlabel $c_i$ at 685 30
\pinlabel $\la''$ at 735 230
\pinlabel $D$ at 740 100
\pinlabel $A''$ at 810 280
\pinlabel $c_{i+1}$ at 870 30
\pinlabel $\mu''$ at 900 155
\endlabellist
\centering
\includegraphics[scale=0.45]{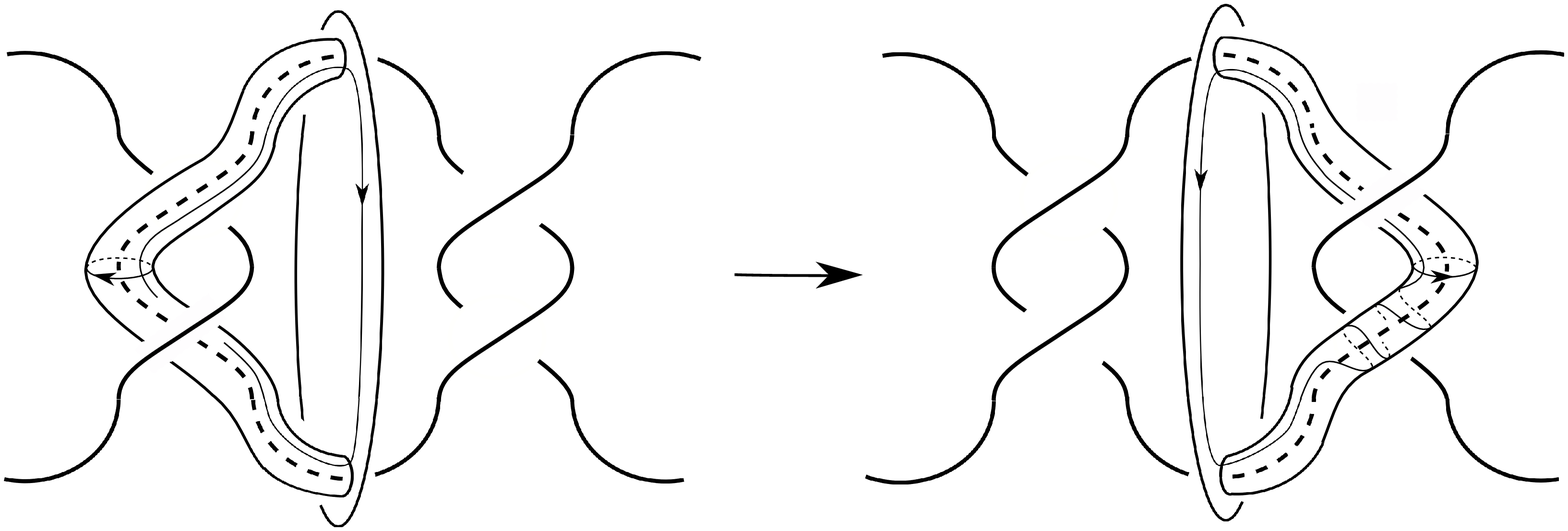}
\caption{The isotopy from $S'$ to $S''$.} 
\label{f:3}
\end{figure}
By composing all the isotopies described so far, we see that $Y$ admits an open book decomposition 
with one--holed torus page, identified with $T$. The corresponding monodromy is obtained by composing 
the diffeomorphisms induced by the various isotopies. We do the calculation in terms of $x$ and $y$, 
starting from the first component $L_1$. Using the above analysis, denoting 
conjugation equivalence by $\sim$ and following our conventions on the composition of diffeomorphisms 
(see paragraph before Proposition~\ref{p:steinfill}) we obtain 
\begin{multline*}
x^{c_1}\phi x^{c_2}\cdots\phi x^{c_N}\phi^{-1} = 
x^{a_1+2}\phi(x^2\phi)^{b_1-1} x^{a_2+2}\cdots x^{a_n+2}\phi(x^2 \phi)^{b_n-1}\phi^{-2}=\\
x^{a_1+1} y^{-b_1}\cdots x^{a_n} y^{-b_n+1} x^2 y x \sim 
x y x^2 y x x^{a_1} y^{-b_1}\cdots x^{a_n} y^{-b_n} = (xy)^3 x^{a_1} y^{-b_1}\cdots x^{a_n} y^{-b_n} = h.
\end{multline*}
\end{proof} 

By Proposition~\ref{p:obd} we can view Figure~\ref{f:1} as a presentation of $Y_{(T,h)}$, 
including the induced open book decomposition. On the other hand, we can also view the framed link 
$L$ as prescribing the attachment of $n$ four--dimensional 2--handles to the 4--ball $B^4$, resulting 
in a smooth oriented 4--manifold $X$ with boundary orientation--preserving diffeomorphic to $Y_{(T,h)}$.
Moreover, note that $L$ is a characteristic sublink of itself, i.e.~$\lk(L, L_i) = \lk(L_i, L_i)\bmod 2$ for 
each component $L_i\subset L$. Recall~\cite{GS99}*{\S 5.7} that there is a natural one--to--one 
correspondence between Spin structures on $Y_{(T,h)}$ and characteristic sublinks of $L$, given by assigning 
to a Spin structure $\Theta$ the sublink $C$ of $L$ consisting of all components $L_i$ such that $\Theta$ does not 
extend across the 2--handle in $X$ attached to $L_i$. Moreover, by~\cite{EO08}*{Lemma~6.1} the Euler class of $\xi_{(T,h)}$ 
vanishes, therefore $\xi_{(T,h)}$ is trivial as a 2--plane bundle over $Y_{(T,h)}$. 
Homotopy classes of trivializations of $\xi_{(T,h)}$ are in 1--1 correspondence with homotopy classes of 
maps $Y_{(T,h)}\to S^1$. If $Y_{(T,h)}$ is a rational homology 3--sphere we have $H^1(Y_{(T,h)};\Z)=0$, 
therefore $\xi_{(T,h)}$ admits a unique trivialization up to homotopy. Moreover, each trivalization of $\xi_{(T,h)}$ canonically 
determines a trivialization of $TY_{(T,h)}$, hence a Spin structure on $Y_{(T,h)}$. We denote by 
$\Theta_\xi$ the Spin structure on $Y_{(T,h)}$ associated in this way to $\xi$. The following lemma 
is an adaptation to the present situation of~\cite{KM94}*{Lemma~A.6}. 

\begin{lem}\label{l:spinstr}
The Spin structure $\Theta_\xi$ corresponds to $L$ viewed as a characteristic sublink of itself. 
\end{lem}

\begin{proof} 
Let $L_{i-1}\subset L$ be 
any component of $L$, and $\mu$ an oriented meridian of $L_{i-1}$ sitting on a page $S$ of 
the open book decomposition of Proposition~\ref{p:obd}, as illustrated in the left--hand side of Figure~\ref{f:2}. 
Since $\xi_{(T,h)}$ is compatible with the open book decomposition, up to homotopy we may assume that 
the trivialization of $TY_{(T,h)}$ associated to a trivialization of $\xi_{(T,h)}$ restricts to $\mu$ as the tangent 
to $\mu$ followed by the normal to $\mu$ in $S$ and the normal to $S$. This framing of $TY_{(T,h)} |_\mu$
has a natural stabilization to a framing of $TX|_\mu$, and as such it does not extend to the cocore of the 
2--handle attached to $L_{i-1}$, therefore $L_{i-1}$ belongs to the characteristic 
sublink corresponding to $\Theta_\xi$. Since the same argument holds for each component of $L$, 
the statement is proved. 
\end{proof} 

For each component $L_i$ of $L$ there is a 2--sphere $S_i$ smoothly embedded in $X$, obtained as the union of a 
2--disc properly embedded in $B^4$ with boundary $L_i$, with the core of the 2--handle attached along 
$L_i$ with framing $c_i$. We fix an orientation of $L$ by orienting each component of $L$ in anti--clockwise fashion 
in the diagram of Figure~\ref{f:1}. This orientation of $L$ prescribes on orientation of each $S_i$ such that, if 
$v_i\in H_2(X;\Z)$ denotes the corresponding 2--homology class, 
the classes $v_1,\ldots, v_N$ form a basis of $H_2(X;\Z)$ and intersect as follows: 
\begin{equation}\label{e:inters}
v_i\cdot v_j = 
\begin{cases} 
c_i &\text{if $i=j$}\\
-1 & \text{if $\{i,j\}\neq\{1,N\}$ and $|i-j|=1$}\\
1 & \text{if $\{i,j\}=\{1,N\}$}
\end{cases} 
\end{equation}
Using this, it is also easy to check that the homology class 
\begin{equation}\label{e:w}
w:=\sum_{i=1}^N v_i\in H_2(X;\Z)
\end{equation}
is characteristic, that is $w\cdot \al \equiv \al\cdot \al\bmod 2$ for every $\al\in H_2(X;\Z)$. 
The following lemma will be used in the proofs of Propositions~\ref{p:isoemb} and~\ref{p:structure}.

\begin{lem}\label{l:posinters}
Let $(\La,\cdot)$ be an intersection lattice of rank $N\geq 2$. Suppose that $v_1,\ldots, v_N$ 
is a basis of $\La$ satisfying~\eqref{e:inters} with $c_1,\ldots, c_N\geq 2$. Then, $(\La,\cdot)$ 
is positive definite. 
\end{lem} 

\begin{proof} 
Let 
\[
\xi=\sum_{i=1}^N x_i v_i\in\La,\quad x_1,\ldots, x_N\in\Z. 
\]
Since $c_1\ldots,c_N\geq 2$, we have  
\begin{eqnarray*}
\xi\cdot\xi = \left(\sum_{i=1}^N x_i v_i\right)\cdot\left(\sum_{i=1}^N x_i v_i\right) = 
\sum_{i=1}^N x_i^2 c_i - 2\sum_{i=1}^{N-1} x_i\cdot x_{i+1} + 2 x_1 x_N\\
\geq 2 \sum_{i=1}^N x_i^2  - 2\sum_{i=1}^{N-1} x_i\cdot x_{i+1} + 2 x_1 x_N \\
= (x_1-x_2)^2+(x_2-x_3)^2+\cdots+(x_{N-1}-x_N)^2+(x_N+x_1)^2\geq 0.
\end{eqnarray*}
Moreover, $\xi\cdot\xi=0$ implies $x_1=\cdots=x_N=-x_1$, i.e.~$\xi=0$. This shows that $(\La,\cdot)$ 
is positive definite. 
\end{proof} 

Denote by $\D_K$ the intersection lattice $(\Z^K, -I)$, i.e.~the standard diagonal negative definite 
intersection lattice of rank $K$. 

\begin{prop}\label{p:isoemb}
Let $h$ be an element of $\Diff^+(T,\del T)$ which can be written as:
\[
h=(xy)^3 x^{a_1} y^{-b_1}\cdots x^{a_n} y^{-b_n}\quad a_i, b_i, n\geq 1.
\] 
If $\xi_{(T,h)}$ is Stein fillable, then there is an isometric embedding of intersection lattices
\[
\varphi: Q_{-X}:=(H_2(-X;\Z),\cdot) \hra\D_K,
\]
where $K = \sum_{i=1}^n a_i + 4$. Morever, $\varphi$ sends the element $w$ of~\eqref{e:w} to a characteristic element. 
\end{prop} 

\begin{proof} 
Given a Stein filling $(W,J)$ of $(Y_{(T,h)},\xi_{(T,h)})$ we can form the smooth, 
closed, oriented 4--manifold $M:=W\cup (-X)$. Proposition~\ref{p:steinfill} and Lemma~\ref{l:posinters} imply 
that the intersection lattice $Q_M:=(H_2(M;\Z),\cdot)$ is negative definite, therefore 
by Donaldson's theorem~\cite{Do87}*{Theorem~1} $Q_M$ is isomorphic to the standard diagonal 
intersection lattice of the same rank: $Q_M\cong \D_K$, where $K=b_2(M)$. Moreover, in view of 
Proposition~\ref{p:steinfill}, we have $b_2(W) = exp(h) - 2 = \sum_{i=1}^n a_i - \sum_{i=1}^n b_i + 4$, therefore 
\[
b_2(M) = b_2(W) + b_2(-X) = exp(h) - 2 + N =  4 + \sum_{i=1}^n a_i.
\]
In particular, there is an isometric embedding $\Q_{-X}\hra\D_K$. This proves the first part of the statement. 
Since the class $w$ defined by~\eqref{e:w} is characteristic, its reduction modulo $2$ is 
represented by a closed surface $\Si_w\subset X$ dual to the second Stiefel--Whitney class $w_2(X)$. 
Then, $X\setminus W$ admits a spin structure $\s$ whose restriction to $\del X=Y_{(T,h)}$ corresponds to 
$L$ viewed as a characteristic sublink of itself, and therefore equals $\Theta_\xi$ 
according to Lemma~\ref{l:spinstr}. But $\Si_w$ can be chosen to be an oriented surface representing 
an integral lift of $w_2(X)$ which is the first Chern class $c_1(\s_w)$ of the unique extension of $\s$ 
to all of $X$ as a Spin$^c$ structure $\s_w$. By construction, the restriction of $\s_w$ to $\del X$ is $\Theta_\xi$. 
Therefore, since the Spin$^c$ structure $\s_J$ on $W$ induced by the 
complex structure also restricts as $\Theta_\xi$ to $\del W=\del(-X)$, there is a Spin$^c$ structure 
$\s_M = \s_J\cup \s_w$ on $M$ whose first Chern class $c_1(\s_M)$ vanishes on $W$ by Proposition~\ref{p:steinfill} 
and restricts to $X$ as $c_1(\s_w)$. This shows that $\varphi(w)$ is the Poincar\'e dual of $c_1(\s_M)$ and therefore 
is characteristic. 
\end{proof}

\section{The proof of Theorems~\ref{t:newmain} and~\ref{t:main}}\label{s:final}

In this section we first derive some crucial consequences from Proposition~\ref{p:isoemb} and then 
we use them to prove Theorems~\ref{t:newmain} and~\ref{t:main}. 

Let $v_1,\ldots, v_N$ be the basis of $Q_{-X}$ chosen as in the previous section and 
satisfying~\eqref{e:inters}. Let $\varphi$ denote an isometric embedding as in Proposition~\ref{p:isoemb}, and denote by 
$\overline w\in \D_K$ the image of $w=\sum_{i=1}^N v_i\in Q_{-X}$ under $\varphi$. The element $\overline w$ 
has the same square as $w$, that is 
\begin{equation}\label{e:wsquare}
\overline w\cdot \overline w = w\cdot w = \sum_{i=1}^n (a_i+2) + 2\sum_{i=1}^n (b_i-1) - 2(\sum_{i=1}^n b_i - 1) + 2 =
4 + \sum_{i=1}^n a_i = K.
\end{equation}
Since $\overline w$ is characteristic, there is a basis 
$e_1,\ldots, e_K\in\D_K$ such that $e_i\cdot e_i = -\de_{ij}$ for every $i, j$  
and $\overline w = \sum_{i=1}^K e_i$. Let $\overline v_1,\ldots, \overline v_N\in\D_K$ 
be the images of, respectively, $v_1,\ldots, v_N$ under $\varphi$. We can define a $K\x N$ matrix $M=(m_{ij})$ by 
expressing each vector $\overline v_j$ in terms of the $e_i$'s:
\begin{equation}\label{e:M}
\overline v_j = \sum_{i=1}^K m_{ij} e_i,\quad j=1,\ldots, N.
\end{equation}
Observe that, since $\overline w = \sum_{i=1}^N \overline v_i$, we have 
\begin{equation}\label{e:whits}
\overline w\cdot \overline v_i = 
\begin{cases} 
\overline v_i \cdot \overline v_i & \text{if $i \in \{1, N\}$},\\
\overline v_i \cdot \overline v_i + 2 &  \text{if $i \not\in \{1, N\}$}.
\end{cases} 
\end{equation}

\begin{lem}\label{l:mij}
Let $M=(m_{ij})$ be the $K\x N$ matrix defined by~\eqref{e:M}.
\begin{enumerate}
\item
For each $j\in\{2,\ldots, N-1\}$ there is a unique index $\tau(j)\in\{1,\ldots, K\}$ such that\\ $m_{\tau(j) j}\in \{-1,2\}$; 
\item
For each $(i,j)\in\{1,\ldots, K\}\x \{1,N\}$ and for each $(i,j)\in\{1,\ldots, K\}\x \{2,\ldots, N-1\}$ with $i\neq\tau(j)$ 
we have $m_{ij}\in\{0,1\}$.
\end{enumerate}
\end{lem} 

\begin{proof} 
Note that for every $(i,j)\in\{1,\ldots, K\}\x \{1,\ldots, N\}$ we have
\begin{equation*}
m_{ij} (m_{ij} - 1) =
\begin{cases} 
0 & \text{if $m_{ij}\in\{0,1\}$},\\
2 & \text{if $m_{ij}\in\{-1, 2\}$}
\end{cases} 
\end{equation*}
and $m_{ij} (m_{ij} - 1)> 2$ if $m_{ij}\not\in\{-1,0,1,2\}$. 
Now, if $j\in\{1,N\}$ by~\eqref{e:whits} 
\[
\sum_{i=1}^N m_{ij}^2 = -\overline v_j\cdot\overline v_j = -\overline w\cdot \overline v_j = 
-\left(\sum_i^K e_i\right)\cdot\left(\sum_i^K m_{ij} e_i\right) = \sum_i^K m_{ij}, 
\]
therefore 
\begin{equation}\label{e:m(m-1)=0}
\sum_{i=1}^K m_{ij} (m_{ij} - 1) = 0.
\end{equation}
Equation~\eqref{e:m(m-1)=0} implies that $m_{ij}\in\{0,1\}$ when $(i,j)\in\{1,\ldots, K\}\x \{1,N\}$. 
Similarly, when $(i,j)\in\{1,\ldots, K\}\x \{2,\ldots, N-1\}$ from~\eqref{e:whits} we obtain
\[
\sum_{i=1}^K m_{ij} (m_{ij} - 1) = 2,
\]
which implies that there is a unique index $\tau(j)$ such that $m_{\tau(j) j}\in \{-1,2\}$, while 
for $i\neq\tau(j)$ we must have $m_{ij}\in\{0,1\}$. 
\end{proof} 

\begin{prop}\label{p:structure}
Let $(c_1,\ldots, c_N)$ be the $N$--tuple defined in~\eqref{e:cstring}, $N\geq 2$. Then, there is a sequence of blowups 
$(s_1,\ldots,s_N)\rightarrow\stackrel{\text{blowup}}{\cdots}\rightarrow (0,0)$ such that 
$c_1\geq s_1,\ c_2\geq s_2,\ \ldots, c_N\geq s_N$.  
\end{prop} 

\begin{proof} 
If $N=2$ there is nothing to prove. Hence, I will assume $N\geq 3$. 
Let $M=(m_{ij})$ be the $K\x N$ matrix defined by~\eqref{e:M}. For each $i=1,\ldots, K$ we have 
\begin{equation}\label{e:summij}
\sum_{j=1}^N m_{ij} = -e_i\cdot \sum_{j=1}^N \overline v_j = -e_i\cdot\overline w = -e_i\cdot\sum_{k=1}^N e_k =1.
\end{equation} 
Note that, by Lemma~\ref{l:mij} and~\eqref{e:summij}, if a row of $M$ contains more than one 
nonzero entry then one of those entries equals $-1$. On the other hand, by Lemma~\ref{l:mij} 
at most $N-2$ entries of $M$ are equal to $-1$. But since $\exp(h) = 2 + K - N \geq 2$, we have 
$K\geq N$, i.e.~the matrix $M$ has at least $N$ rows. We conclude that a row $R$ of $M$ has a single 
nonzero entry, which by~\eqref{e:summij} must be equal to $1$. Deleting $R$ from $M$ we obtain a new 
matrix $M' = (m'_{ij})$ having $K-1$ rows and $N$ columns. The $m'_{ij}$'s still satisfy~\eqref{e:summij}. Moreover, 
we can use the $m'_{ij}$'s as in~\eqref{e:M} to define elements $\overline v'_j\in\D_{K-1}$, $j=1,\ldots, N$. 
Note that $\overline w' := \sum_{j=1}^N \overline v'_j$  and the $\overline v'_j$'s intersect as in~\eqref{e:whits}. 
Therefore the proof of Lemma~\ref{l:mij} goes trough and, since $K-1\geq N-1 > N-2$, we can reapply the argument 
just used to conclude that a row of $M'$ has a single nonzero entry equal to $1$. Then we can delete that row 
obtaining a new matrix, reapply the same argument and 
so on. If we keep going this way until we can, i.e.~until we obtain a matrix $M''$ with $N-2$ rows and $N$ columns,  
the elements $\overline v''_1,\ldots, \overline v''_N\in\D_{N-2}$ defined by the columns of $M'' = (m''_{ij})$ will satisfy 
\begin{equation}\label{e:intpatt}
v''_i\cdot v''_j = 
\begin{cases} 
-1 & \text{if $\{i,j\}=\{1,N\}$}\\
1 & \text{if $\{i,j\}\neq\{1,N\}$ and $|i-j|=1$}.
\end{cases} 
\end{equation}
In view of Lemma~\ref{l:posinters} we must necessarily have $\overline v''_j\cdot \overline v''_j = -1$ 
for some $j\in\{1,\ldots, N\}$. In particular $m''_{ij} =\pm 1$ for some $i$ and $m''_{sj}=0$ for $s\neq i$. 
Erasing the $i$--th row and the $j$--th column of $M''$ we get a matrix whose columns define $N-1$ 
elements of $\D_{N-3}$ which intersect as in~\eqref{e:intpatt}. Since we are assuming $N\geq 3$, 
we can keep going in the same way until we have three elements in $\D_1$ intersecting each other 
in the usual way and having all square $-1$. Reconstructing backwards the various steps it is easy 
to check that $\overline v''_1,\ldots, \overline v''_N\in\D_{N-2}$ have self--intersections $(-s_1,\ldots, -s_N)$, 
with $(s_1,\ldots,s_N)\rightarrow\stackrel{\text{blowup}}{\cdots}\rightarrow (0,0)$, 
and that $c_i\geq s_i$ for each $i=1,\ldots, N$. This concludes the proof. 
\end{proof} 

\begin{proof}[Proof of Theorems~\ref{t:newmain} and~\ref{t:main}]
If $N=1$ then $n=b_1=1$ and $h$ is clearly positive, therefore we may assume $N\geq 2$. 
As we recalled in Section~\ref{s:intro}, the fact that $(1)\Rightarrow (2)$ is well known. 
Moreover, by Proposition~\ref{p:3implies1}  we have the implication $(3)\Rightarrow (1)$ of Theorem~\ref{t:newmain} 
and by Proposition~\ref{p:structure} we have $(2)\Rightarrow (3)$. This concludes the proof 
of Theorem~\ref{t:newmain}, which together with Proposition~\ref{p:steinfill} implies Theorem~\ref{t:main}. 
\end{proof} 

\begin{rem}\label{r:highergenus}
As pointed out by John Etnyre~\cite{Epc13}, the fact that there exist Stein fillable, non--positive open books $(\Si, h)$ of any genus $g(\Si)\geq 2$ can be proved as follows.  Let $(\Si_1, f_1)$ be any positive open book with $\Si_1$ equal to
a one--holed torus, and let $(\Si_2, f_2)$ be one of the Stein fillable, non--positive examples from~\cite{Wa09} or~\cite{BEV10}
with $g(\Si_2)=2$. Consider a boundary connected sum $(\Si = \Si_1\natural\Si_2, f_1\natural f_2)$ (it doesn't matter which component of $\del\Si_2$ is involved in the sum).  Then, $\xi_{(\Si, f_1\natural f_2)}$ is the contact connected sum of $\xi_{(\Si_1,f_1)}$ and $\xi_{(\Si_2,f_2)}$ and therefore $(\Si, f_1\natural f_2)$ is Stein fillable. Moreover, there is a properly embedded arc $a\subset\Si$ with endpoints on the same boundary component $C$ of $\del\Si$, such that the an open neighborhood $N$ of $a\cup C$ in $\Si$ is a pair of pants whose complement  is homeomorphic to the disjoint union of $\Si_1$ and $\Si_2$. By construction $f_1\natural f_2$ has a representative which restricts to $N$ as the identity. Suppose by contradiction that $f_1\natural f_2$ can be written as a composition of right--handed Dehn twists $\de_{C_1}\circ\cdots\circ\de_{C_k}$. It is easy to check that if $C_i\cap a\not=\emptyset$ for some 
$i\in\{1,\ldots,k\}$, then the arc $a$ is sent by $f_1\natural f_2$ ``to the right''  in the sense of~\cite{HKM07}. On the other hand, 
by construction $f_1\natural f_2(a) = a$, which is not to the right of $a$. This implies that $C_i\cap a=\emptyset$ 
for each $i$, and therefore $f_1\natural f_2 = P_1\natural P_2$, where $P_i:\Si_i\to\Si_i$, $i=1,2$, is a positive diffeomorphism. 
But the map $(f_1, f_2)\mapsto f_1\natural f_2$ is a group homomorphism, thus applying~\cite{PR00}*{Corollary~4.2 (iii)} one can easily show that it is injective. We conclude that $f_2$ is positive, contrary to the initial assumption. Repeating the same  
construction sufficiently many times one can construct Stein fillable, non--positive open books with pages of 
any genus strictly bigger than one. 
\end{rem} 

\bibliography{biblio}

\end{document}